\documentclass[12pt,twoside]{amsart}
\usepackage{amsmath, amsthm, amscd, amsfonts, amssymb, graphicx}

\usepackage{enumerate}
\usepackage[colorlinks=true,
linkcolor=blue,
urlcolor=cyan,
citecolor=red]{hyperref}
\usepackage{mathrsfs}
\newtheorem{theorem}{Theorem}[section]

\newtheorem{corollary}{Corollary}[section]

\numberwithin{equation}{section}

\begin{document}
	
\title{matrix Fej\'{e}r and Levin-Ste\v{c}kin inequalities}
\author{Mohammad Sababheh$^{1}$, Shiva Sheybani$^{2}$ and Hamid Reza Moradi$^{3}$}

\begin{abstract}
Fe{j}\'er and Levin-Ste\v{c}kin inequalities treat integrals of the product of convex functions with symmetric functions. The main goal of this article is to present possible matrix versions of these inequalities. In particular, majorization results are shown of Fej\'{e}r type for both convex and log-convex functions. For matrix Levin-Ste\v{c}kin type, we present  more rigorous results involving the partial Loewner ordering for Hermitian matrices.
\end{abstract}
\maketitle
%------------------------------------------------------------------------------------%
\pagestyle{myheadings}
\markboth{\centerline {}}
{\centerline {}}
\bigskip
\bigskip
%------------------------------------------------------------------------------------%
%------------------------------------------------------------------------------------%
\section{Introduction}
The theory of Convex functions has played a major role in the study of Mathematical inequalities. Related to convex-type inequalities, 
the  Levin-Ste\v{c}kin's inequality states that if the function $p:\left[ 0,1 \right]\to \mathbb{R}$ is symmetric about $t=\frac{1}{2}$, namely $p\left( 1-t \right)=p\left( t \right),$  and non-decreasing  on $\left[ 0,\frac{1}{2} \right]$, then for every convex function $f$ on $[0,1]$, the inequality
\begin{equation}\label{3}
\int\limits_{0}^{1}{p\left( t \right)f\left( t \right)dt}\le \int\limits_{0}^{1}{p\left( t \right)dt}\int\limits_{0}^{1}{f\left( t \right)dt}
\end{equation}
holds true \cite{levin}.
If $p$ is symmetric non-negative (without any knowledge about its monotonicity) and $f$ is convex, Fej\'er inequality states that \cite{fejer}
\[\int\limits_{0}^{1}{p\left( t \right)dt}f\left( \frac{a+b}{2} \right)\le \int\limits_{0}^{1}{p\left( t \right)f\left( \left( 1-t \right)a+tb \right)dt}\le \int\limits_{0}^{1}{p\left( t \right)dt}\frac{f\left( a \right)+f\left( b \right)}{2}.\]
We notice that Fej\'er inequality reduces to the Hermite-Hadamard inequality \cite{hadamard}, when $p(t)=1.$
In the field of Mathematical inequalities, it is of interest to extend known inequalities from the setting of scalars to other objects; such as matrices. In this article, we will be interested in extending both the Levin-Ste\v{c}kin and Fej\'er inequalities to the matrices setting.

In the sequel, $\mathcal{M}_n$ will denote the algebra of all $n\times n$ complex matrices. The conjugate transpose (or adjoint) of $A\in\mathcal{M}_n$ is denoted by $A^*$, and then the matrix $A$ will be called Hermitian if $A^*=A.$ When $\left<Ax,x\right>\geq 0$ for all $x\in\mathbb{C}^n$, $A$ is said to be positive semidefinite, and is denoted as $A\geq 0.$ If $A\geq 0$ and $A$ is invertible, then $A$ is said to be positive (strictly positive or positive definite). When $A,B\in\mathcal{M}_n$ are Hermitian, we say that $A\leq B$ if $B-A\geq 0.$ This provides a partial ordering on the class of Hermitian matrices. The eigenvalues of a Hermitian matrix $A$ will be denoted by $\lambda_1(A),\lambda_2(A),\cdots,\lambda_n(A)$, repeated according to their multiplicity and arranged decreasingly. That is $\lambda_1(A)\geq\lambda_2(A)\geq\cdots\geq\lambda_n(A).$ 

The relation $A\leq B$ implies $\lambda_i(A)\leq \lambda_i(B)$ for any such Hermitian matrices $A,B\in\mathcal{M}_n$. However, the converse is not true. This urges the need to discuss, in some cases, the latter order. For convenience, we will write $\lambda(A)\leq \lambda(B)$ to mean that $\lambda_i(A)\leq\lambda(B), i=1,2,\cdots,n.$

Another weaker ordering among matrices is the so called weak majorization $\prec_w$, defined for the Hermitian matrices $A,B$ as
$$A\prec_w B\Leftrightarrow \sum_{i=1}^{k}\lambda_i(A)\leq\sum_{i=1}^{k}\lambda_i(B), k=1,\cdots,n.$$
It is clear that \cite{bhatia}
$$A\leq B\Rightarrow \lambda(A)\leq \lambda(B)\Rightarrow A\prec_w B.$$
It is customary to obtain one of these orders when extending a scalar inequality to a matrix inequality. For example, in this article we obtain 
\[\lambda \left( \int\limits_{0}^{1}{p\left( t \right)dt}f\left( \frac{A+B}{2} \right) \right)\le \lambda \left( \int\limits_{0}^{1}{p\left( t \right)f\left( \left( 1-t \right)A+tB \right)dt} \right),\]
as an extension of Fe{j}\'er inequality, to the Hermitian matrices $A,B$ with spectra in the domain of $f$.

Further, if $f$ is monotone, then
\[\lambda \left( \int\limits_{0}^{1}{p\left( t \right)f\left( \left( 1-t \right)A+tB \right)dt} \right)\le \lambda \left( \int\limits_{0}^{1}{p\left( t \right)dt}\frac{f\left( A \right)+f\left( B \right)}{2} \right);\]
as matrix inequalities of the Fej\'er inequality.

In the next section we study the possible matrix versions of Fej\'er inequality, which in turns imply certain versions of the Hermite-Hadamard matrix inequality \cite{moslehian}. Then log-convex functions will be deployed to obtain new matrix Fej\'er inequalities for this type of functions, and we conclude with the discussion of the matrix Levin-Ste\v{c}kin inequality.

\section{Fej\'er matrix inequalities for convex functions}
We begin with the following weak majorization of  Fej\'er-type inequality. 
\begin{theorem}\label{thm1}
Let $f:J\to\mathbb{R}$ be convex and let $p:[0,1]\to [0,\infty)$ be symmetric about $t=\frac{1}{2}.$ If $A,B\in\mathcal{M}_n$ are Hermitian with spectra in the interval $J$, then
\[\lambda \left( \int\limits_{0}^{1}{p\left( t \right)dtf\left( \frac{A+B}{2} \right)} \right){{\prec }_{w}}\lambda \left( \int\limits_{0}^{1}{p\left( t \right)f\left( \left( 1-t \right)A+tB \right)dt} \right).\]
\end{theorem}
 \begin{proof}
 If $f$ is a convex function, then for any $0 \le t \le 1$
\[\begin{aligned}
   f\left( \frac{a+b}{2} \right)&=f\left( \frac{\left( 1-t \right)a+tb+\left( 1-t \right)b+ta}{2} \right) \\ 
 & \le \frac{f\left( \left( 1-t \right)a+tb \right)+f\left( \left( 1-t \right)b+ta \right)}{2}.  
\end{aligned}\]
Thus,
\begin{equation}\label{1}
f\left( \frac{a+b}{2} \right)\le \frac{f\left( \left( 1-t \right)a+tb \right)+f\left( \left( 1-t \right)b+ta \right)}{2}.
\end{equation}
If the function $p:\left[ 0,1 \right]\to \mathbb{R}$ is symmetric, we get from \eqref{1},
\[p\left( t \right)f\left( \frac{a+b}{2} \right)\le p\left( t \right)\left( \frac{f\left( \left( 1-t \right)a+tb \right)+f\left( \left( 1-t \right)b+ta \right)}{2} \right).\]
Integrating on $\left[ 0,1 \right]$, we get 
\begin{equation}\label{2}
\int\limits_{0}^{1}{p\left( t \right)}dtf\left( \frac{a+b}{2} \right)\le \int\limits_{0}^{1}{p\left( t \right)f\left( \left( 1-t \right)a+tb \right)dt}.
\end{equation}
If we replace $a$, $b$ by $\left\langle Ax,x \right\rangle $, $\left\langle Bx,x \right\rangle $ respectively, in \eqref{2}, we get 
\begin{equation}\label{4}
\int\limits_{0}^{1}{p\left( t \right)dtf\left( \frac{\left\langle Ax,x \right\rangle +\left\langle Bx,x \right\rangle }{2} \right)}\le \int\limits_{0}^{1}{p\left( t \right)f\left( \left( 1-t \right)\left\langle Ax,x \right\rangle +t\left\langle Bx,x \right\rangle  \right)dt}.
\end{equation}
On the other hand, it follows from the Jensen's inequality \cite[Theorem 1.2]{mond-pecaric},
	\[f\left( \left\langle \left( \left( 1-t \right)A+tB \right)x,x \right\rangle  \right)\le \left\langle f\left( \left( 1-t \right)A+tB \right)x,x \right\rangle .\]
By multiplying both sides by $p\left( t \right)$, we get
	\[p\left( t \right)f\left( \left\langle \left( \left( 1-t \right)A+tB \right)x,x \right\rangle  \right)\le p\left( t \right)\left\langle f\left( \left( 1-t \right)A+tB \right)x,x \right\rangle .\]
Therefore,
\begin{equation}\label{5}
\int\limits_{0}^{1}{p\left( t \right)f\left( \left\langle \left( \left( 1-t \right)A+tB \right)x,x \right\rangle  \right)dt}\le \left\langle \left( \int\limits_{0}^{1}{p\left( t \right)f\left( \left( 1-t \right)A+tB \right)dt} \right)x,x \right\rangle.
\end{equation}
Combining inequalities \eqref{4} with \eqref{5}, we obtain
\begin{equation}\label{6}
\int\limits_{0}^{1}{p\left( t \right)dtf\left( \frac{\left\langle Ax,x \right\rangle +\left\langle Bx,x \right\rangle }{2} \right)}\le \left\langle \left( \int\limits_{0}^{1}{p\left( t \right)f\left( \left( 1-t \right)A+tB \right)dt} \right)x,x \right\rangle .
\end{equation}
Suppose that ${{\lambda }_{1}},\ldots ,{{\lambda }_{m}}$ are the eigenvalues of $\frac{A+B}{2}$ with ${{x}_{1}},\ldots ,{{x}_{m}}$ as an orthonormal system of corresponding eigenvectors arranged such that $f\left( {{\lambda }_{1}} \right)\ge \cdots \ge f\left( {{\lambda }_{m}} \right)$. We have
\[\begin{aligned}
   \sum\limits_{j=1}^{k}{{{\lambda }_{j}}\left( \int\limits_{0}^{1}{p\left( t \right)dtf\left( \frac{A+B}{2} \right)} \right)}&=\sum\limits_{j=1}^{k}{\int\limits_{0}^{1}{p\left( t \right)dtf\left( \left\langle \left( \frac{A+B}{2} \right){{x}_{j}},{{x}_{j}} \right\rangle  \right)}} \\ 
 & =\sum\limits_{j=1}^{k}{\int\limits_{0}^{1}{p\left( t \right)dtf\left( \frac{\left\langle A{{x}_{j}},{{x}_{j}} \right\rangle +\left\langle B{{x}_{j}},{{x}_{j}} \right\rangle }{2} \right)}} \\ 
 & \le \sum\limits_{j=1}^{k}{\left\langle \left( \int\limits_{0}^{1}{p\left( t \right)f\left( \left( 1-t \right)A+tB \right)dt} \right){{x}_{j}},{{x}_{j}} \right\rangle } \\ 
 &\qquad \text{(by the inequality \eqref{6})}\\
 & \le \sum\limits_{j=1}^{k}{{{\lambda }_{j}}\left( \int\limits_{0}^{1}{p\left( t \right)f\left( \left( 1-t \right)A+tB \right)dt} \right)}.  
\end{aligned}\]
Namely,
\[\sum\limits_{j=1}^{k}{{{\lambda }_{j}}\left( \int\limits_{0}^{1}{p\left( t \right)dtf\left( \frac{A+B}{2} \right)} \right)}\le \sum\limits_{j=1}^{k}{{{\lambda }_{j}}\left( \int\limits_{0}^{1}{p\left( t \right)f\left( \left( 1-t \right)A+tB \right)dt} \right)}.\]
Therefore,
\[\lambda \left( \int\limits_{0}^{1}{p\left( t \right)dtf\left( \frac{A+B}{2} \right)} \right){{\prec }_{w}}\lambda \left( \int\limits_{0}^{1}{p\left( t \right)f\left( \left( 1-t \right)A+tB \right)dt} \right).\]
 \end{proof}

\section{Fej\'er inequalities via log-convex functions}
In this part of the paper, we show a matrix Fej\'er inequality for log-convex functions. 
\begin{theorem}\label{thm_log}
Let $f$ be log-convex and $p:[0,1]\to (0,\infty)$ be symmetric and normalized in the sense that $\int_{0}^{1}p(t)dt=1.$ Then
\[\lambda\left(\log f\left(\frac{A+B}{2}\right)\right)\prec_w\lambda\left(\log\int_{0}^{1}p(t)f((1-t)A+tB)dt\right).\]
\end{theorem}
\begin{proof}
When $f$ is convex, we have
\[f\left(\left<\frac{A+B}{2}x,x\right>\right)\leq \int_{0}^{1}p(t)f\left<((1-t)A+tB)x,x\right>dt,\] for any unit vector $x$. Since $f$ is log-convex, it follows that
\[\log f\left(\left<\frac{A+B}{2}x,x\right>\right)\leq \int_{0}^{1}p(t)\log f\left<((1-t)A+tB)x,x\right>dt.\] Noting that $\log$ is a concave function and that $d\mu(t):=p(t)dt$ is a probability measure, we have
\begin{align*}
\log f\left(\left<\frac{A+B}{2}x,x\right>\right)&\leq \int_{0}^{1}p(t)\log f\left<((1-t)A+tB)x,x\right>dt\\
&=\int_{0}^{1}\log f\left<((1-t)A+tB)x,x\right>d\mu(t)\\
&\leq \log\int_{0}^{1} f\left<((1-t)A+tB)x,x\right>d\mu(t)\\
&=\log\int_{0}^{1} p(t)f\left<((1-t)A+tB)x,x\right>dt,
\end{align*}
for any unit vector $x$. Now let $\lambda_j$ be the eigenvalues of $\frac{A+B}{2}$ with orthonormal eigenvectors $x_1,x_2,\cdots,x_n$, so that $f(\lambda_1)\geq\cdots\geq f(\lambda_n).$ Then
\begin{align*}
\sum_{j=1}^{k}\lambda_j\left(\log f\left(\frac{A+B}{2}\right)\right)&=\sum_{j=1}^{k} \log f(\lambda_j)\\
&=\sum_{j=1}^{k}\log f\left(\left<\frac{A+B}{2}x_j,x_j\right>\right)\\
&\leq \sum_{j=1}^{k}\log\int_{0}^{1} p(t)f\left<((1-t)A+tB)x_j,x_j\right>dt\\
&\leq \sum_{j=1}^{k}\lambda_j\left(\log\int_{0}^{1}p(t)f((1-t)A+tB))dt\right).
\end{align*}
This completes the proof.
\end{proof}
 As a consequence, we have the following.
 
 \begin{corollary}
 Let $f$ be log-convex and $p:[0,1]\to (0,\infty)$ be symmetric and normalized. Then
 \[\prod_{j=1}^{k}\lambda_j\left(f\left(\frac{A+B}{2}\right)\right)\leq  \prod_{j=1}^{k}\lambda_j\left(\int_{0}^{1}p(t)f((1-t)A+tB)dt\right), k=1,\cdots,n.\]
 \end{corollary} 
 \begin{proof}
 From Theorem\ref{thm_log}, we have
 \begin{align*}
 \sum_{j=1}^{k}\lambda_j\left(\log f\left(\frac{A+B}{2}\right)\right)&\leq 
 \sum_{j=1}^{k}\lambda_j\left(\log\int_{0}^{1}p(t)f((1-t)A+tB))dt\right)\\
 \Rightarrow \sum_{j=1}^{k}\log\lambda_j\left( f\left(\frac{A+B}{2}\right)\right)&\leq 
 \sum_{j=1}^{k}\log\lambda_j\left(\int_{0}^{1}p(t)f((1-t)A+tB))dt\right)\\
 \Rightarrow \log\prod_{j=1}^{k}\lambda_j\left(f\left(\frac{A+B}{2}\right)\right)&\leq  \log\prod_{j=1}^{k}\lambda_j\left(\int_{0}^{1}p(t)f((1-t)A+tB)dt\right),
 \end{align*}
 which implies the desired inequality.
 \end{proof}

 \section{Levin-Ste\v{c}kin matrix inequalities}
 We begin by presenting a new  inequality of Levin-Ste\v{c}kin type. The significance of this inequality is its validity for any positive function $p$; without imposing any conditions on its symmetry or monotony.
 \begin{theorem}\label{thm_general}
 Let $f:[0,1]\to\mathbb{R}$ be convex differentiable and let $p:[0,1]\to[0,\infty)$ be continuous. Then
 \begin{align*}
 \int_{0}^{1}f(t)dt\int_{0}^{1}p(t)dt+\left(\int_{0}^{1}f'(t)dt\int_{0}^{1}tp(t)dt-\int_{0}^{1}tf'(t)dt\int_{0}^{1}p(t)dt\right)\leq\int_{0}^{1}f(t)p(t)dt.
 \end{align*}
 Further,
 \[\int\limits_{0}^{1}{p\left( t \right)f\left( t \right)dt}+\frac{1}{2}\int\limits_{0}^{1}{p\left( t \right)f'\left( t \right)dt}-\int\limits_{0}^{1}{p\left( t \right)tf'\left( t \right)dt}\le \int\limits_{0}^{1}{p\left( t \right)dt}\int\limits_{0}^{1}{f\left( t \right)dt}.\]

 \end{theorem}
 \begin{proof}
 For the convex differentiable function $f$ and $s,t\in [0,1]$ we have
 \begin{equation}\label{8}
f(s)+f'(s)(t-s)\leq f(t).
 \end{equation}
Since $p(t)\geq 0,$ it follows that
 $$p(t)f(s)+p(t)f'(s)(t-s)\leq p(t)f(t), s,t\in [0,1].$$
 Integrating this inequality over $t\in [0,1]$ then over $s\in[0,1]$ implies
 \begin{align*}
  \int_{0}^{1}f(s)dt\int_{0}^{1}p(t)dt+\left(\int_{0}^{1}f'(s)dt\int_{0}^{1}tp(t)dt-\int_{0}^{1}sf'(s)dt\int_{0}^{1}p(t)dt\right)\leq\int_{0}^{1}f(t)p(t)dt,
 \end{align*}
which is equivalent to the first desired inequality.

For the second inequality,  integrating \eqref{8} over $t\in \left[ 0,1 \right]$, we obtain 
	\[f\left( s \right)+f'\left( s \right)\left( \frac{1}{2}-s \right)\le \int\limits_{0}^{1}{f\left( t \right)dt}.\]
If we put $s=t$, we have
	\[f\left( t \right)+f'\left( t \right)\left( \frac{1}{2}-t \right)\le \int\limits_{0}^{1}{f\left( t \right)dt}.\]
Multiplying both sides by $p\left( t \right)$, we get
	\[p\left( t \right)f\left( t \right)+p\left( t \right)f'\left( t \right)\left( \frac{1}{2}-t \right)\le p\left( t \right)\int\limits_{0}^{1}{f\left( t \right)dt}.\]
Again, if we take integral over $t\in \left[ 0,1 \right]$ we infer that
\[\int\limits_{0}^{1}{p\left( t \right)f\left( t \right)dt}+\frac{1}{2}\int\limits_{0}^{1}{p\left( t \right)f'\left( t \right)dt}-\int\limits_{0}^{1}{p\left( t \right)tf'\left( t \right)dt}\le \int\limits_{0}^{1}{p\left( t \right)dt}\int\limits_{0}^{1}{f\left( t \right)dt}.\]
This completes the proof.
 \end{proof}

 \begin{corollary}
 Let $f:[0,1]\to\mathbb{R}$ be convex differentiable and let $p:[0,1]\to[0,\infty)$  be symmetric about $\frac{1}{2}$ and non-decreasing on $\left[0,\frac{1}{2}\right].$ Then
 \begin{align*}
 \int_{0}^{1}f'(t)dt\int_{0}^{1}tp(t)dt\leq \int_{0}^{1}tf'(t)dt\int_{0}^{1}p(t)dt.
 \end{align*}

 \end{corollary}
 \begin{proof}
 This follows from the first inequality in Theorem \ref{thm_general} because when $p$ is symmetric about $1/2$ and non-decreasing on $\left[0,\frac{1}{2}\right]$, we have $$\int_{0}^{1}f(t)p(t)dt\leq \int_{0}^{1}f(t)dt\int_{0}^{1}p(t)dt.$$
 \end{proof}

 The following is the operator Levin-St\v{e}ckin inequality, see also \cite[Theorem 2]{dragomir1}.
 \begin{theorem}
 Let $f$ be operator convex and $p:[0,1]\to [0,\infty)$ be symmetric about $t=\frac{1}{2}$ and non-decreasing on $\left[0,\frac{1}{2}\right].$ Then
  \[\int\limits_{0}^{1}{p\left( t \right)f\left( \left( 1-t \right)A+tB \right)dt}\le \int\limits_{0}^{1}{p\left( t \right)dt}\int\limits_{0}^{1}{f\left( \left( 1-t \right)A+tB \right)dt}.\]
 \end{theorem}
 \begin{proof}
  Let $x\in\mathbb{C}^n$ be a unit vector. Since the function $F\left( t \right)=\left\langle f\left( \left( 1-t \right)A+tB \right)x,x \right\rangle $ is a real-valued convex function on $\left[ 0,1 \right]$ (see \cite[Theorem 1]{dragomir}), we have  
\[\begin{aligned}
   \left\langle \int\limits_{0}^{1}{p\left( t \right)f\left( \left( 1-t \right)A+tB \right)dt}x,x \right\rangle &=\int\limits_{0}^{1}{p\left( t \right)\left\langle f\left( \left( 1-t \right)A+tB \right)x,x \right\rangle dt} \\ 
 & \le \int\limits_{0}^{1}{p\left( t \right)dt}\int\limits_{0}^{1}{\left\langle f\left( \left( 1-t \right)A+tB \right)x,x \right\rangle }dt \\ 
 & =\left\langle \int\limits_{0}^{1}{p\left( t \right)dt}\int\limits_{0}^{1}{f\left( \left( 1-t \right)A+tB \right)dt}x,x \right\rangle .  
\end{aligned}\]
Therefore,
  \[\int\limits_{0}^{1}{p\left( t \right)f\left( \left( 1-t \right)A+tB \right)dt}\le \int\limits_{0}^{1}{p\left( t \right)dt}\int\limits_{0}^{1}{f\left( \left( 1-t \right)A+tB \right)dt}.\]
 \end{proof}
 
 The following theorem gives a reverse for the operator Levin-St\v{e}ckin's inequality by employing the Mond- Pe\v cari\'c method \cite{mond-pecaric}.
 \begin{theorem}
  Let $f:\left[ m,M \right]\to \mathbb{R}$ be  convex and let $p:[0,1]\to [0,\infty)$ be symmetric about $t=\frac{1}{2}.$ If $A,B\in\mathcal{M}_n$ are Hermitian with spectra in the interval $\left[ m,M \right]$, then for any $\alpha \ge 0$
\[\int\limits_{0}^{1}{p\left( t \right)dt}\int\limits_{0}^{1}{f\left( \left( 1-t \right)A+tB \right)dt}\le \beta \int\limits_{0}^{1}{p\left( t \right)dt}I+\alpha \int\limits_{0}^{1}{p\left( t \right)f\left( \left( 1-t \right)A+tB \right)dt},\]
where $\beta =\underset{m\le x\le M}{\mathop{\max }}\,\left\{ {{a}_{f}}x+{{b}_{f}}-\alpha f\left( x \right) \right\}$, ${{a}_{f}}=\frac{f\left( M \right)-f\left( m \right)}{M-m}$, and ${{a}_{f}}=\frac{Mf\left( m \right)-mf\left( M \right)}{M-m}$.
  \end{theorem}
  \begin{proof}
Since
\[f\left( x \right)\le {{a}_{f}}x+{{b}_{f}},\]
we get by the functional calculus 
\[f\left( \left( 1-t \right)A+tB \right)\le {{a}_{f}}\left( \left( 1-t \right)A+tB \right)+{{b}_{f}}I.\]
By taking integral over $0\le t\le 1$ , we reach to 
\[\int\limits_{0}^{1}{f\left( \left( 1-t \right)A+tB \right)dt}\le {{a}_{f}}\left( \frac{A+B}{2} \right)+{{b}_{f}}I.\] 
This implies, 
\[\int\limits_{0}^{1}{p\left( t \right)dt}\int\limits_{0}^{1}{f\left( \left( 1-t \right)A+tB \right)dt}\le \int\limits_{0}^{1}{p\left( t \right)dt}{{a}_{f}}\left( \frac{A+B}{2} \right)+\int\limits_{0}^{1}{p\left( t \right)dt}{{b}_{f}}I.\] 
Hence for any vector $x$,                                                                                                                   \[\left\langle \left( \int\limits_{0}^{1}{p\left( t \right)dt}\int\limits_{0}^{1}{f\left( \left( 1-t \right)A+tB \right)dt} \right)x,x \right\rangle \le \int\limits_{0}^{1}{p\left( t \right)dt}{{a}_{f}}\left\langle \left( \frac{A+B}{2} \right)x,x \right\rangle +\int\limits_{0}^{1}{p\left( t \right)dt}{{b}_{f}}.\]
Now, by \eqref{4}, we can write 
                                                                                      \[\begin{aligned} 
  & \left\langle \left( \int\limits_{0}^{1}{p\left( t \right)dt}\int\limits_{0}^{1}{f\left( \left( 1-t \right)A+tB \right)dt} \right)x,x \right\rangle -\alpha \int\limits_{0}^{1}{p\left( t \right)f\left( \left\langle \left( \left( 1-t \right)A+tB \right)x,x \right\rangle  \right)dt} \\ 
 & \le \int\limits_{0}^{1}{p\left( t \right)dt}{{a}_{f}}\left\langle \left( \frac{A+B}{2} \right)x,x \right\rangle +\int\limits_{0}^{1}{p\left( t \right)dt}{{b}_{f}}-\alpha \int\limits_{0}^{1}{p\left( t \right)f\left( \left\langle \left( \left( 1-t \right)A+tB \right)x,x \right\rangle  \right)dt} \\ 
 & \le \int\limits_{0}^{1}{p\left( t \right)dt}{{a}_{f}}\left\langle \left( \frac{A+B}{2} \right)x,x \right\rangle +\int\limits_{0}^{1}{p\left( t \right)dt}{{b}_{f}}-\alpha \int\limits_{0}^{1}{p\left( t \right)dt}f\left( \left\langle \left( \frac{A+B}{2} \right)x,x \right\rangle  \right) \\ 
 & =\int\limits_{0}^{1}{p\left( t \right)dt}\left( {{a}_{f}}\left\langle \left( \frac{A+B}{2} \right)x,x \right\rangle +{{b}_{f}}-\alpha f\left( \left\langle \left( \frac{A+B}{2} \right)x,x \right\rangle  \right) \right) \\ 
 & \le \int\limits_{0}^{1}{p\left( t \right)dt}\underset{m\le x\le M}{\mathop{\max }}\,\left\{ {{a}_{f}}x+{{b}_{f}}-\alpha f\left( x \right) \right\}. \\ 
\end{aligned}\] 
Thus, 
\[\begin{aligned} 
  & \left\langle \left( \int\limits_{0}^{1}{p\left( t \right)dt}\int\limits_{0}^{1}{f\left( \left( 1-t \right)A+tB \right)dt} \right)x,x \right\rangle  \\ 
 & \le \int\limits_{0}^{1}{p\left( t \right)dt}\beta +\alpha \int\limits_{0}^{1}{p\left( t \right)f\left( \left\langle \left( \left( 1-t \right)A+tB \right)x,x \right\rangle  \right)dt} \\ 
 & \le \int\limits_{0}^{1}{p\left( t \right)dt}\beta +\alpha \int\limits_{0}^{1}{p\left( t \right)\left\langle f\left( \left( 1-t \right)A+tB \right)x,x \right\rangle dt} \\ 
 &\qquad \text{(by \cite[Theorem 1.2]{mond-pecaric})}\\
 & =\left\langle \left( \int\limits_{0}^{1}{p\left( t \right)dt}\beta I+\alpha \int\limits_{0}^{1}{p\left( t \right)f\left( \left( 1-t \right)A+tB \right)dt} \right)x,x \right\rangle  \\ 
\end{aligned}\] 
as desired.
  \end{proof}
  
  \section{Further inequalities via synchronous functions}
  We say that the functions $f,g:J\to \mathbb{R}$ are synchronous (asynchronous) on the interval $J$ if they satisfy the following condition:
  	\[\left( f\left( t \right)-f\left( s \right) \right)\left( g\left( t \right)-g\left( s \right) \right)\ge \left( \le  \right)0,\text{ }\forall s,t\in J.\]
  It is obvious that, if $f, g$ are monotonic and have the same monotonicity
  on the interval $J$, then they are synchronous on $J$ while if they have
  opposite monotonicity, they are asynchronous.
  
Related to the Levin-Ste\v{c}kin inequality,   the celebrated \^Ceby\^sev inequality \cite{ceb} states that if $f$ and $g$ are two functions having the same monotonicity on $\left[ 0,1 \right]$, then
  \[\int\limits_{0}^{1}{f\left( t \right)dt}\int\limits_{0}^{1}{g\left( t \right)dt}\le \int\limits_{0}^{1}{f\left( t \right)g\left( t \right)dt}.\]
  The following result provides a refinement and a reverse of this inequality, via synchronous functions.
  \begin{theorem}\label{01}
  Let $f,g\text{:}\left[ a,b \right]\to \mathbb{R}$ be synchronous functions on the interval $\left[ a,b \right]$. Then
  \[\begin{aligned}
    & \min \left\{ \frac{1}{b-a}\int\limits_{a}^{b}{{{f}^{2}}\left( t \right)dt}-{{\left( \frac{1}{b-a}\int\limits_{a}^{b}{f\left( t \right)dt} \right)}^{2}},\frac{1}{b-a}\int\limits_{a}^{b}{{{g}^{2}}\left( t \right)dt}-{{\left( \frac{1}{b-a}\int\limits_{a}^{b}{g\left( t \right)dt} \right)}^{2}} \right\} \\ 
   & \le \frac{1}{b-a}\int\limits_{a}^{b}{f\left( t \right)g\left( t \right)dt}-\frac{1}{b-a}\int\limits_{a}^{b}{f\left( t \right)dt}\frac{1}{b-a}\int\limits_{a}^{b}{g\left( t \right)dt} \\ 
   & \max \left\{ \frac{1}{b-a}\int\limits_{a}^{b}{{{f}^{2}}\left( t \right)dt}-{{\left( \frac{1}{b-a}\int\limits_{a}^{b}{f\left( t \right)dt} \right)}^{2}},\frac{1}{b-a}\int\limits_{a}^{b}{{{g}^{2}}\left( t \right)dt}-{{\left( \frac{1}{b-a}\int\limits_{a}^{b}{g\left( t \right)dt} \right)}^{2}} \right\}.  
  \end{aligned}\]
If $f$ and $g$ have the opposite monotonicity then
\[\begin{aligned}
  & \min \left\{ \frac{1}{b-a}\int\limits_{a}^{b}{{{f}^{2}}\left( t \right)dt}-{{\left( \frac{1}{b-a}\int\limits_{a}^{b}{f\left( t \right)dt} \right)}^{2}},\frac{1}{b-a}\int\limits_{a}^{b}{{{g}^{2}}\left( t \right)dt}-{{\left( \frac{1}{b-a}\int\limits_{a}^{b}{g\left( t \right)dt} \right)}^{2}} \right\} \\ 
 & \le \frac{1}{b-a}\int\limits_{a}^{b}{f\left( t \right)dt}\frac{1}{b-a}\int\limits_{a}^{b}{g\left( t \right)dt}-\frac{1}{b-a}\int\limits_{a}^{b}{f\left( t \right)g\left( t \right)dt} \\ 
 & \le \max \left\{ \frac{1}{b-a}\int\limits_{a}^{b}{{{f}^{2}}\left( t \right)dt}-{{\left( \frac{1}{b-a}\int\limits_{a}^{b}{f\left( t \right)dt} \right)}^{2}},\frac{1}{b-a}\int\limits_{a}^{b}{{{g}^{2}}\left( t \right)dt}-{{\left( \frac{1}{b-a}\int\limits_{a}^{b}{g\left( t \right)dt} \right)}^{2}} \right\}.  
\end{aligned}\]
   \end{theorem}
  \begin{proof}
  We prove the first inequality. The second inequality goes likewise and we omit the details. We have
  \[\begin{aligned}
    & f\left( t \right)g\left( t \right)+f\left( s \right)g\left( s \right)-\left( f\left( t \right)g\left( s \right)+f\left( s \right)g\left( t \right) \right) \\ 
   & =\left( f\left( t \right)-f\left( s \right) \right)\left( g\left( t \right)-g\left( s \right) \right) \\ 
   & =\left| \left( f\left( t \right)-f\left( s \right) \right)\left( g\left( t \right)-g\left( s \right) \right) \right| \\ 
   & =\left| f\left( t \right)-f\left( s \right) \right|\left| g\left( t \right)-g\left( s \right) \right| \\ 
   & \ge \min \left\{ {{\left( f\left( t \right)-f\left( s \right) \right)}^{2}},{{\left( g\left( t \right)-g\left( s \right) \right)}^{2}} \right\} \\ 
   & =\min \left\{  {{f}^{2}}\left( t \right)+{{f}^{2}}\left( s \right)-2f\left( t \right)f\left( s \right), {{g}^{2}}\left( t \right)+{{g}^{2}}\left( s \right)-2g\left( t \right)g\left( s \right)  \right\}.  
  \end{aligned}\]
      Therefore,
\[\begin{aligned}
  & \min \left\{ {{f}^{2}}\left( s \right)+{{f}^{2}}\left( t \right)-2f\left( s \right)f\left( t \right),{{g}^{2}}\left( t \right)+{{g}^{2}}\left( s \right)-2g\left( t \right)g\left( s \right) \right\} \\ 
 & \le f\left( t \right)g\left( t \right)+f\left( s \right)g\left( s \right)-\left( f\left( t \right)g\left( s \right)+f\left( s \right)g\left( t \right) \right).
\end{aligned}\]
     Consequently, 
\[\begin{aligned}
  & \min \left\{ \left( b-a \right){{f}^{2}}\left( s \right)+\int\limits_{a}^{b}{{{f}^{2}}\left( t \right)dt}-2f\left( s \right)\int\limits_{a}^{b}{f\left( t \right)dt},\int\limits_{a}^{b}{{{g}^{2}}\left( t \right)dt}+\left( b-a \right){{g}^{2}}\left( s \right)-2g\left( s \right)\int\limits_{a}^{b}{g\left( t \right)dt} \right\} \\ 
 & \le \int\limits_{a}^{b}{f\left( t \right)g\left( t \right)dt}+\left( b-a \right)f\left( s \right)g\left( s \right)-g\left( s \right)\int\limits_{a}^{b}{f\left( t \right)dt}-f\left( s \right)\int\limits_{a}^{b}{g\left( t \right)dt}.  
\end{aligned}\]
 Upon integration, this implies  
\[\begin{aligned}
  & \min \left\{ 2\left( b-a \right)\int\limits_{a}^{b}{{{f}^{2}}\left( t \right)dt}-2{{\left( \int\limits_{a}^{b}{f\left( t \right)dt} \right)}^{2}},2\left( b-a \right)\int\limits_{a}^{b}{{{g}^{2}}\left( t \right)dt}-2{{\left( \int\limits_{a}^{b}{g\left( t \right)dt} \right)}^{2}} \right\} \\ 
 & \le 2\left( b-a \right)\int\limits_{a}^{b}{f\left( t \right)g\left( t \right)dt}-2\int\limits_{a}^{b}{f\left( t \right)dt}\int\limits_{a}^{b}{g\left( t \right)dt}. \\ 
\end{aligned}\]
Multiplying both sides by ${1}/{2{{\left( b-a \right)}^{2}}}\;$, we obtain,
\[\begin{aligned}
  & \min \left\{ \frac{1}{b-a}\int\limits_{a}^{b}{{{f}^{2}}\left( t \right)dt}-{{\left( \frac{1}{b-a}\int\limits_{a}^{b}{f\left( t \right)dt} \right)}^{2}},\frac{1}{b-a}\int\limits_{a}^{b}{{{g}^{2}}\left( t \right)dt}-{{\left( \frac{1}{b-a}\int\limits_{a}^{b}{g\left( t \right)dt} \right)}^{2}} \right\} \\ 
 & \le \frac{1}{b-a}\int\limits_{a}^{b}{f\left( t \right)g\left( t \right)dt}-\frac{1}{b-a}\int\limits_{a}^{b}{f\left( t \right)dt}\frac{1}{b-a}\int\limits_{a}^{b}{g\left( t \right)dt}. 
\end{aligned}\]
The second inequality obtains from the same arguments and the following relation
\[\begin{aligned}
  & \max \left\{ {{f}^{2}}\left( s \right)+{{f}^{2}}\left( t \right)-2f\left( s \right)f\left( t \right),{{g}^{2}}\left( t \right)+{{g}^{2}}\left( s \right)-2g\left( t \right)g\left( s \right) \right\} \\ 
 & \ge f\left( t \right)g\left( t \right)+f\left( s \right)g\left( s \right)-\left( f\left( t \right)g\left( s \right)+f\left( s \right)g\left( t \right) \right). \\ 
\end{aligned}\]

  \end{proof}

We establish a refinement and a reverse for the Levin-Ste\v{c}kin inequality in the next result. 
\begin{theorem}
Let $p:\left[ 0,1 \right]\to \mathbb{R}$ be a symmetric about $t=\frac{1}{2}$, namely $p\left( 1-t \right)=p\left( t \right),$  and non-decreasing  on $\left[ 0,\frac{1}{2} \right]$, then for every convex function $f$ on $[0,1]$, 
\[\begin{aligned}
  & \int\limits_{0}^{1}{p\left( t \right)f\left( t \right)dt}\le \int\limits_{0}^{1}{p\left( t \right)dt}\int\limits_{0}^{1}{f\left( t \right)dt} \\ 
 & -\min \left\{ 2\int\limits_{0}^{{1}/{2}\;}{{{p}^{2}}\left( t \right)dt}-{{\left( \int\limits_{0}^{1}{p\left( t \right)dt} \right)}^{2}},\frac{1}{2}\int\limits_{0}^{{1}/{2}\;}{{{\left( {f\left( t \right)+f\left( 1-t \right)} \right)}^{2}}dt}-{{\left(\frac{1}{2} \int\limits_{0}^{1}{({f\left( t \right)+f\left( 1-t \right)})dt} \right)}^{2}} \right\}.
\end{aligned}\]
A similar but reversed inequality holds if we replace $\min $ with $\max $.
\end{theorem}
\begin{proof}
If if $f$ is symmetric and convex, by Theorem \ref{01}, we have
  \[\begin{aligned}
    & \int\limits_{0}^{1}{p\left( t \right)dt}\int\limits_{0}^{1}{f\left( t \right)dt} \\ 
   & =\left( \int\limits_{0}^{{1}/{2}\;}{p\left( t \right)dt}+\int\limits_{{1}/{2}\;}^{1}{p\left( t \right)dt} \right)\left( \int\limits_{0}^{{1}/{2}\;}{f\left( t \right)dt}+\int\limits_{{1}/{2}\;}^{1}{f\left( t \right)dt} \right) \\ 
   & =4\int\limits_{0}^{{1}/{2}\;}{p\left( t \right)dt}\int\limits_{0}^{{1}/{2}\;}{f\left( t \right)dt} \\ 
   & \ge 2\int\limits_{0}^{{1}/{2}\;}{p\left( t \right)f\left( t \right)dt}+\min \left\{ 2\int\limits_{0}^{{1}/{2}\;}{{{p}^{2}}\left( t \right)dt}-{{\left( 2\int\limits_{0}^{{1}/{2}\;}{p\left( t \right)dt} \right)}^{2}},2\int\limits_{0}^{{1}/{2}\;}{{{f}^{2}}\left( t \right)dt}-{{\left( 2\int\limits_{0}^{{1}/{2}\;}{f\left( t \right)dt} \right)}^{2}} \right\} \\ 
   & =\int\limits_{0}^{1}{p\left( t \right)f\left( t \right)dt}+\min \left\{ 2\int\limits_{0}^{{1}/{2}\;}{{{p}^{2}}\left( t \right)dt}-{{\left( \int\limits_{0}^{1}{p\left( t \right)dt} \right)}^{2}},2\int\limits_{0}^{{1}/{2}\;}{{{f}^{2}}\left( t \right)dt}-{{\left( \int\limits_{0}^{1}{f\left( t \right)dt} \right)}^{2}} \right\}.  
  \end{aligned}\]
Namely,  
  \[\begin{aligned}
    & \int\limits_{0}^{1}{p\left( t \right)f\left( t \right)dt}+\min \left\{ 2\int\limits_{0}^{{1}/{2}\;}{{{p}^{2}}\left( t \right)dt}-{{\left( \int\limits_{0}^{1}{p\left( t \right)dt} \right)}^{2}},2\int\limits_{0}^{{1}/{2}\;}{{{f}^{2}}\left( t \right)dt}-{{\left( \int\limits_{0}^{1}{f\left( t \right)dt} \right)}^{2}} \right\} \\ 
   & \le \int\limits_{0}^{1}{p\left( t \right)dt}\int\limits_{0}^{1}{f\left( t \right)dt}.
  \end{aligned}\]
We shall consider now an arbitrary $f$.  
  For convex $f$ the function $\frac{f\left( x \right)+f\left( 1-x \right)}{2}$ is convex and symmetric, so we can use the above inequality. Hence,
  	\[\begin{aligned}
    & \int\limits_{0}^{1}{p\left( t \right)f\left( t \right)dt} \\ 
   & =\frac{\int_{0}^{1}{p\left( t \right)f\left( t \right)dt}+\int_{0}^{1}{p\left( 1-t \right)f\left( 1-t \right)dt}}{2} \\ 
   & =\int\limits_{0}^{1}{p\left( t \right)\frac{f\left( t \right)+f\left( 1-t \right)}{2}dt} \\ 
   & \le \int\limits_{0}^{1}{p\left( t \right)dt}\int\limits_{0}^{1}{\frac{f\left( t \right)+f\left( 1-t \right)}{2}dt} \\ 
   & -\min \left\{ 2\int\limits_{0}^{{1}/{2}\;}{{{p}^{2}}\left( t \right)dt}-{{\left( \int\limits_{0}^{1}{p\left( t \right)dt} \right)}^{2}},\frac{1}{2}\int\limits_{0}^{{1}/{2}\;}{{{\left( {f\left( t \right)+f\left( 1-t \right)} \right)}^{2}}dt}-{{\left(\frac{1}{2} \int\limits_{0}^{1}{({f\left( t \right)+f\left( 1-t \right)})dt} \right)}^{2}} \right\} \\ 
   & = \int\limits_{0}^{1}{p\left( t \right)dt}\int\limits_{0}^{1}{f\left( t \right)dt} \\ 
   & -\min \left\{ 2\int\limits_{0}^{{1}/{2}\;}{{{p}^{2}}\left( t \right)dt}-{{\left( \int\limits_{0}^{1}{p\left( t \right)dt} \right)}^{2}},\frac{1}{2}\int\limits_{0}^{{1}/{2}\;}{{{\left( {f\left( t \right)+f\left( 1-t \right)} \right)}^{2}}dt}-{{\left(\frac{1}{2} \int\limits_{0}^{1}{({f\left( t \right)+f\left( 1-t \right)})dt} \right)}^{2}} \right\},
  \end{aligned}\]
which yields the desired inequality.  
\end{proof}

  We can improve the second inequality in Theorem \ref{01}, in the following way.
    \begin{theorem}
    Let $f,g:J\to \mathbb{R}$ be synchronous functions on the interval $\left[ 0,1 \right]$. Then
    \[\begin{aligned}
      & \int\limits_{0}^{1}{f\left( t \right)g\left( t \right)dt}-\int\limits_{0}^{1}{f\left( t \right)dt}\int\limits_{0}^{1}{g\left( t \right)dt} \\ 
     & \le \frac{1}{2}\left( \int\limits_{0}^{1}{{{f}^{2}}\left( t \right)dt}-{{\left( \int\limits_{0}^{1}{f\left( t \right)dt} \right)}^{2}}+\int\limits_{0}^{1}{{{g}^{2}}\left( t \right)dt}-{{\left( \int\limits_{0}^{1}{g\left( t \right)dt} \right)}^{2}} \right).
    \end{aligned}\]
    \end{theorem}
    \begin{proof}
    We have
    \[\begin{aligned}
      & f\left( t \right)g\left( t \right)+f\left( s \right)g\left( s \right)-\left( f\left( t \right)g\left( s \right)+f\left( s \right)g\left( t \right) \right) \\ 
     & =\left( f\left( t \right)-f\left( s \right) \right)\left( g\left( t \right)-g\left( s \right) \right) \\ 
     & =\left| \left( f\left( t \right)-f\left( s \right) \right)\left( g\left( t \right)-g\left( s \right) \right) \right| \\ 
     & =\left| f\left( t \right)-f\left( s \right) \right|\left| g\left( t \right)-g\left( s \right) \right| \\ 
     & \le \frac{1}{2}\left( {{\left( f\left( t \right)-f\left( s \right) \right)}^{2}}+{{\left( g\left( t \right)-g\left( s \right) \right)}^{2}} \right) \\ 
     & =\frac{1}{2}\left( {{f}^{2}}\left( t \right)+{{f}^{2}}\left( s \right)+{{g}^{2}}\left( t \right)+{{g}^{2}}\left( s \right)-2\left( g\left( t \right)g\left( s \right)+f\left( t \right)f\left( s \right) \right) \right). 
    \end{aligned}\]
    Therefore,
    \[\begin{aligned}
      & f\left( t \right)g\left( t \right)+f\left( s \right)g\left( s \right)-\left( f\left( t \right)g\left( s \right)+f\left( s \right)g\left( t \right) \right) \\ 
     & \le \frac{1}{2}\left( {{f}^{2}}\left( t \right)+{{f}^{2}}\left( s \right)+{{g}^{2}}\left( t \right)+{{g}^{2}}\left( s \right)-2\left( g\left( t \right)g\left( s \right)+f\left( t \right)f\left( s \right) \right) \right).  
    \end{aligned}\]
The remaining part of the proof is similar to the proof of Theorem \ref{01}, so we omit the details
    \end{proof}
  
  \section*{Acknowledgement}
  The authors would like to thank Prof. J. C. Bourin, who pointed out a crucial mistake in the previous version of the manuscript.

{\tiny $^{1}$Department of basic sciences, Princess Sumaya University for Technology, Amman 11941, Jordan}

{\tiny \textit{E-mail address:} sababheh@psut.edu.jo}

{\tiny $^{2}$Department of Mathematics, Islamic Azad University, Mashhad Branch, Mashhad, Iran}

{\tiny \textit{E-mail address:} shiva.sheybani95@gmail.com}

{\tiny $^{3}$Department of Mathematics, Payame Noor University (PNU), P.O. Box 19395-4697, Tehran, Iran}

{\tiny \textit{E-mail address:} hrmoradi@mshdiau.ac.ir}
%-----------------------------------------------------------------------------
%-----------------------------------------------------------------------------
\end{document}